\documentclass[12pt]{article}

\usepackage{amsmath,amsthm,amsfonts,amscd,amssymb,graphicx,latexsym,mathrsfs}

\title{Generalized Solid Angle Theory for Real Polytopes}

\author{David DeSario and Sinai Robins}

\def\real{{\mathbb R}}
\def\zed{{\mathbb Z}}
\def\C{{\mathbb C}}

\def\P{{\mathcal{P}}}
\def\K{{\mathcal{K}}}
\def\F{{\mathcal{F}}}
\def\G{{\mathcal{G}}}
\def\l{{\mathit{l}}}
\def\m{{\mathit{m}}}
\def\x{{\mathbf{x}}}
\def\v{{\mathbf{v}}}

\def\w{{\mathbf{w}}}
\def\y{{\mathbf{y}}}

\def\ep{{\epsilon}}

\newtheorem{thm}{Theorem}[section]
\newtheorem{cor}[thm]{Corollary}

\newtheorem{ex}{Example}[section]

\newtheorem{conj}{Conjecture}
\newtheorem{fact}{Fact}
\theoremstyle{definition}

\theoremstyle{remark}

\begin{document}
\maketitle

\section{Introduction}

The natural generalization of a two-dimensional angle to higher dimensions is called a \emph{solid angle}. Given a pointed cone $\K \subset \real^d$, the solid angle at its apex is the proportion of space that the cone $\K$ occupies.  Alternatively, a solid angle can be thought of as the volume of a spherical polytope.     Ian Macdonald initiated the systematic study of solid-angle sums in integral polytopes with his 1971 paper \cite{Macdonald} and currently there is a resurgence of activity on solid angles; see, for example,  \cite{Camenga}, \cite{Ribando}, and
\cite{Sato}.    The theory of solid angles of polyhedra, which parallels that of integer-point enumeration known as Ehrhart theory, can be found in Chapter 11 of \cite{Sinai} by Beck and Robins.    Macdonald's solid angle sums give us a new measure of discrete volume, and they find applications in the Ehrhart theory of polytopes.

\medskip
In this paper, we extend many theorems from \cite{Sinai}, which hold true for for  \emph{rational}  polytopes, to results for \emph{real} polytopes that  also involve more general solid angles.  A rational polytope is a polytope whose vertices all have rational coordinates, whereas a real polytope is a polytope whose vertices all have arbitrary real coordinates.     Our generalized solid angles are defined using the $l^p$-norm, and in particular include the $l^1$-norm, which gives solid angles that are themselves polyhedral and hence easily computable.

\medskip
The proofs we give here rely on Harmonic Analysis and therefore do not resemble the proofs in \cite{Sinai}, which are combinatorial in nature. Furthermore, it is the power of Harmonic Analysis that allows us to extend our results to all real convex polytopes $\P$ and to all real dilations of $\P$.

\medskip
We note that solid-angle theory for real polytopes is still in its infancy, primarily due to the considerable increase in difficulty associated with the study of polyhedra with irrational vertices. Precise enumeration theorems for real polytopes are hard to come by,  and thus the main contribution of this paper to solid-angle theory is the extension of several fundamental theorems to \emph{real} polytopes.

\section{Definitions and background material}

A \textbf{convex polytope} $\P \subseteq \real^d$ is the bounded intersection of finitely many half-spaces and hyperplanes.  If $\P$ is of dimension $d$, we call it a $d$-polytope.  A \textbf{convex cone} $\K \subseteq \real^d$ is the intersection of finitely many half-spaces of the form $\{\textbf{x} \in \real^d | \, \textbf{a} \cdot \textbf{x} \leq b\}$ whose corresponding hyperplanes $\{\textbf{x} \in \real^d | \, \textbf{a} \cdot \textbf{x} = b\}$ meet in at least one point. A cone is called \textbf{pointed} if the defining hyperplanes meet in exactly one point.  Throughout this paper, the word cone will always refer to a pointed cone.

Suppose $\P \subset \real^d$ is a convex $d$-polytope. Then the \textbf{solid angle} $\omega_{\P}(\mathbf{x})$  of a point $\x$ (with respect to $\P$) equals the proportion of a small ball centered at $\mathbf{x}$ that is contained in $\P$. Thus, for all positive $\ep$ sufficiently small,
\begin{equation}\label{usualsolid}
 \omega_{\P} (\mathbf{x})= \frac{\textrm{vol} (B_{\ep} (\mathbf{x})
  \cap \P)}{\textrm{vol}\, B_{\ep} (\mathbf{x})},   
\end{equation}
where $B_{\ep}(\x)$ is the ball of radius $\ep$ centered at $\x$.  We now generalize our definition of a solid angle by considering balls with respect to $l^p$-norm for any $p \geq 1$.
Given $\mathbf{x}= (x_1, x_2, \dots, x_d) \in \real^d$, the $l^p$-norm of $\mathbf{x}$ is defined by
\[
\| \, \mathbf{x}\|_p = \left(| \, x_1|^p+| \, x_2|^p+\dots+| \, x_d|^p \right)^{1/p}, \ \ \ \textrm{for}\  p \geq 1.
\]
The ball with respect to $l^p$-norm of radius $\ep$ centered at $\mathbf{x}$ is the set
\[
B_{p, \, \ep}(\mathbf{x}) := \{ \mathbf{y} \in \real^d : \ \ \| \, \mathbf{x-y} \|_p < \ep \ \}.
\]
For any convex $d$-polytope $\P \subset \real^d$, the \textbf{\boldmath $l^p$-solid angle} of a point $\x$, denoted by $\omega_{p, \, \P}(\mathbf{x})$, is the proportion of a small $l^p$-ball centered at $\x$ that is contained in $\P$. That is
\begin{equation}\label{refinedsolid}
\omega_{p, \, \P} (\x)= \frac{\textrm{vol} (B_{p, \, \ep} (\x)  \cap \P)}{\textrm{vol}\, B_{p, \, \ep} (\x)},
\end{equation}
for all positive $\ep$ sufficiently small.

Given a cone $\K \subset \real^d$, we also have the following integral definition of a general solid angle with respect to $\K$.  For $\x \in \real^d$ and $p \geq 1$, the \textbf{\boldmath $l^p$-solid angle} of $\x$ with respect to $\K$ is given by

\large
\begin{eqnarray}\label{analyticomega}
\omega_{p, \, \K} (\x) &:=& \lim_{\epsilon \to 0^+} \frac{1}{\epsilon^{d/p}} \int_{\K} e^{\frac{-\,c}{\ep} \| t - \x \|_p^p} dt \label{analyticomega} \\
\nonumber &=& \lim_{\epsilon \to 0^+} \frac{1}{\epsilon^{d/p}} \int_{\K} e^{\frac{-\,c}{\ep} \left( |t_1 -x_1|^p + |t_2 -x_2|^p + \cdots +|t_d -x_d|^p \right)} dt,
\end{eqnarray}
\normalsize
where $c = \left(2 \, \Gamma \left( \frac{1}{p}+1 \right) \right)^p$.  This definition arises from centering at $\x$ a Gaussian function with respect to the $l^p$-norm that is normalized to have a total mass of 1 and then integrating to calculate the proportion of mass contained in $\K$. This definition of $\omega_{p, \, \K} (\x)$ is more analytic in nature, as opposed to geometric, and it opens the door to the Harmonic Analysis techniques that will be used below.

For $\ep >0$, $p \geq 1$, and $t \in \real^d$ we define
\large
\begin{equation}\label{phi_ep}
\phi_{\ep} (t) := \frac{1}{\epsilon^{d/p}} \, e^{\frac{-\,c}{\ep} \| t \|_p^p}. 
\end{equation}
\normalsize
Notice that $\phi_{\ep}(-t) = \phi_{\ep}(t)$, by the properties of the $l^p$-norm, so that
the integral in equation (\ref{analyticomega}) becomes a convolution as follows:

\begin{eqnarray*}
\omega_{p, \, \K} (\x) &=& \lim_{\epsilon \to 0} \int_{\K} \phi_{\ep} (t - \x)dt \\ &=& \lim_{\epsilon \to 0} \int_{\K} \phi_{\ep} (\x - t)dt \\
&=& \lim_{\epsilon \to 0} \int_{\real^d} 1_{\K}(t) \, \phi_{\ep} (\x - t)dt \\ &=& \lim_{\epsilon \to 0} \left( 1_{\K} \ast \phi_{\ep} \right) (\x).
\end{eqnarray*}
The last equality follows from the definition of the convolution.  This fact will be used a great deal, so we highlight it here:
\begin{fact}\label{fact1}
\[
 \omega_{p, \, \K} (\x) =  \lim_{\epsilon \to 0} \left( 1_{\K} \ast \phi_{\ep} \right) (\x), \ \ \textrm{for all } \ \x \in \real^d.
\]
\end{fact}


The \textbf{integer-point transform} of a polytope $\P \subset \real^d$, given by
\begin{equation}
\sigma_{\P}(\mathbf{z}):= \sum_{m \in \P \cap \zed^d} \mathbf{z}^{m},
\end{equation}
is a multivariate generating function that lists all integer points in $\P$ as a formal sum of monomials.  This special format encodes information about the integer points in a way that allows us to use both algebraic and analytic techniques to study the discrete geometry of polyhedra.  By analogy, we form the \textbf{solid-angle generating function} for a polytope $\P$
\begin{equation}
    \alpha_{\P}(\mathbf{z}):= \sum_{m \in \, \P \cap \, \zed^d} \omega_\P (m) \mathbf{z}^{m},
\end{equation}
where $\omega_\P(m)$ is the usual solid angle, defined by (\ref{usualsolid}).

In order to employ the methods of Harmonic Analysis, we often need to consider functions of a complex variable.  For this reason, we redefine $\alpha_\P$ using the substitution $z_k = e^{2 \pi i s_k}$ for each $k= 1, \dots, d$,  so  $\mathbf{z}^{m} = e^{2 \pi i \langle s, \, m\rangle}$ and we obtain
\begin{equation}\label{alpha}
    \alpha_{\P}(s):= \sum_{m \in \, \P \cap \, \zed^d} \omega_\P (m) e^{2 \pi i \langle s, \, m\rangle}, \ \ \ \textrm{for } s \in \C^d.
\end{equation}
This substitution will prove essential when we use the following Poisson summation formula:  If $f$ is a ``sufficiently nice'' function (for example, a function which is $L^1$ and continuous, and has a Fourier transform which is also $L^1$ and continuous), then
\begin{equation}
\sum_{\l \in \, \zed^d} f(\l) = \sum_{\m \in \, \zed^d } \hat{f}(\m),
\end{equation}
where $\hat{f}: \real^d \to \C$ is defined as
\[
\hat{f}(y) = \int_{\real^d} e^{2 \pi i \langle x, \, y \rangle} f(x)dx.
\]
Using this technique will introduce sums of Fourier-Laplace transforms defined over polyhedra and the complex variable will ensure convergence of such sums.  We note that while defined similarly, the Fourier-Laplace transform is defined for the complex variable $s \in \C^d$, while the Fourier transform is only defined on $\real^d$.

We wish to point out that $\alpha_{\P}(s)$ is a finite sum for any polytope $\P \subset \real^d$ and for all $s \in \C^d$ because the $ \omega_\P (m)=0$ for all $m \notin \P$.  Therefore convergence is not an issue when dealing with polytopes.  However, when we consider the solid-angle generating function for a pointed cone $\K$, convergence does become a slight issue.  

To discuss the convergence of $\alpha_{\K}(s)$, we need to define $K^*$, the polar cone associated with $K$. The polar cone $K^*$ is defined by
\[
K^* = \{x \in \real^d \, : \, \langle \, x, \, y \rangle < 0, \forall \, y \in \K \}.
\]
Thus, $\alpha_{\K}(s)$ converges if $s \in \C^d$ such that $-$Im$(s) \in \K^*$, because
\begin{eqnarray*}
&& -\textrm{Im}(s) \in \K^*
\\&\Rightarrow&
\langle -\textrm{Im}(\, s), \, m\rangle < 0, \ \ \forall \ m \in  \K \cap \zed^d
\\&\Leftrightarrow&
\big|e^{2 \pi \langle -\textrm{Im}(\, s), \, m\rangle} \big| < 1, \ \ \forall \ m \in  \K \cap \zed^d
\\&\Leftrightarrow&
\big|e^{2 \pi i \langle s, \, m\rangle} \big| < 1, \ \ \forall \ m \in  \K \cap \zed^d.
\end{eqnarray*}

We now further extend our definition of $\alpha_\P$, by replacing $\omega_\P (m)$ by the generalized $l^p$-solid angle measure   $\omega_{p, \P}(m)$                  
defined by (\ref{analyticomega}) and which we restate here:

\begin{equation}\label{restateanalyticomega}
\omega_{p,\P} (\x): = \lim_{\epsilon \to 0} \frac{1}{\epsilon^{d/p}} \int_{\P} e^{\frac{-\,c}{\ep} \| t - \x \|_p^p} dt.
\end{equation}
Thus, we will always assume that our solid angles $\omega_\P (m)$ are 
in fact the generalized $l^p$-solid angles  $\omega_{p, \P}(m)$, with a fixed
real $p \geq 1$.

Recalling Fact \ref{fact1}, we can write 
$
 \omega_\P (m) = \lim_{\epsilon \to 0} \left(1_{\P} \ast \phi_{\epsilon} \right)(m),
$
with the specific choice of
 $ \phi_{\ep} (t) := \frac{1}{\epsilon^{d/p}} \, e^{\frac{-\,c}{\ep} \| t \|_p^p}$.
We see that the usual definition of a solid angle is retrieved by setting $p=2$.
In general, for any $p$, our $\phi_\epsilon$ enjoys the property that its Fourier-Laplace
transform is rapidly decreasing.   The fact that  $\hat{\phi_\ep}$ decreases rapidly 
at infinity assures us the  absolute convergence of all  the series that ensue.

\section{A functional equation for the generalized solid-angle function $\alpha_K(s)$ of a real cone $K$}

We now show that the solid-angle generating function $\alpha_{\K}(s)$ obeys the following functional equation, also known as a reciprocity relation:
%
%
\begin{thm}
Suppose $\K$ is a real, simple $d$-cone in $\real^d$ with vertex at the origin and $s \in \C^d$.  Then
\begin{equation}
\alpha_{\K}(-s)= (-1)^d \alpha_{\K}(s).
\end{equation}
\end{thm}
\begin{proof}
For $j= 1, \dots, d$, let $\w_j$ be a generator of the simple cone $\K$. By
abuse of notation, we denote the determinant of the matrix whose
$j^{\textrm{th}}$ column is the edge vector $\w_j$ by $\det \K$. Then
\begin{eqnarray}
\alpha_{\K}(-s)&=& \lim_{\epsilon \to 0} \sum_{m \in \zed^d} \left(1_{\K} \ast
\phi_{\epsilon} \right)(m) e^{2 \pi i \langle -s, m \rangle}\\
&=& \lim_{\epsilon \to 0} \sum_{m \in \zed^d} \widehat{\left(1_{\K} \ast
\phi_{\epsilon} \right)}(m - s) \\ &=& \lim_{\epsilon \to 0} \sum_{m \in
\zed^d} \hat{1}_{\K}(m - s) \hat{\phi_{\epsilon}}(m - s) \\ &=& \lim_{\ep \to
0} \sum_{m \in \zed^d} \frac{ (-2 \pi i )^{- \, d} |\det{\K}|}{\prod_{j=1}^{d}
\langle
\w_j, m-s \rangle} \  \hat{\phi_{\epsilon}}(m - s)
\end{eqnarray}
The last equality uses the formula for $\hat{1}_\K$, which is exercise 10.4 in
\cite{Sinai}. We used Poisson summation in the second equality, which is valid
because the convolution of $1_\K$ with $\phi_\epsilon$ is an integrable and
continuous function whenever $\phi_\epsilon$ is integrable and continuous.

Now we will use the fact that the lattice sum is invariant under the
substitution $m = - n$.  Thus, we have
\begin{eqnarray}
\alpha_{\K}(-s)&=& \lim_{\epsilon \to 0} \sum_{n \in \zed^d} \frac{(-2 \pi i )^{- \, d} |\det{\K}|}{\prod_{j=1}^{d} \langle \w_j, -n-s \rangle} \
\hat{\phi_{\epsilon}}(-n - s) \\ &=& (-1)^d \lim_{\epsilon \to 0} \sum_{n \in
\zed^d} \frac{(-2 \pi i )^{- \, d} |\det{\K}|}{\prod_{j=1}^{d} \langle \w_j, n+s \rangle} \ \hat{\phi_{\epsilon}}(n + s) \label{secondequation}\\
&=& (-1)^d \lim_{\epsilon \to 0} \sum_{n \in \zed^d}  \hat{1}_{\K}(n + s)
\hat{\phi_{\epsilon}}(n + s) \\ &=& (-1)^d \lim_{\epsilon \to 0} \sum_{n \in
\zed^d} \widehat{\left(1_{\K} \ast \phi_{\ep} \right)}(n + s) \\
&=& (-1)^d \lim_{\ep \to 0} \sum_{n \in \zed^d} \left(1_{\K} \ast
\phi_{\ep} \right)(n)e^{2 \pi i \langle s, n \rangle} \\ &=& (-1)^d
\alpha_{\K}(s).
\end{eqnarray}
In (\ref{secondequation}), we used the fact that, for all complex vectors $z \in \C^d, \hat{\phi_\epsilon} (-z) = \hat{\phi_\epsilon} (z)$.  This last remark holds because
\begin{eqnarray*}
\hat{\phi_\epsilon} (-z) &=& \int_{\real^d} e^{2 \pi i \langle -z, x \rangle}
\phi_\epsilon (x) dx \\ &=& \int_{\real^d} e^{2 \pi i \langle z, -x \rangle}
\phi_\epsilon (x) dx \\ &=& \int_{\real^d} e^{2 \pi i \langle z, u \rangle}
\phi_\epsilon (-u) du \\ &=& \int_{\real^d} e^{2 \pi i \langle z, u \rangle}
\phi_\epsilon (u) du \\ &=& \hat{\phi_\epsilon} (z).
\end{eqnarray*}

\end{proof}

\noindent We now generalize the previous theorem to any real $d$-cone.

\begin{thm}
Suppose $\K$ is a $d$-cone with its vertex at the origin, $\v \in \real^d$, and $s \in \C^d$. Then the solid-angle generating function $\alpha_{\v + \K} (s)$ of the $d$-cone $\v + \K$ satisfies
\begin{equation}
\alpha_{\v + \K}(-s)= (-1)^d \alpha_{-v + \K}(s).
\end{equation}
\end{thm}

\begin{proof}
Since solid angles are additive, it suffices to prove this theorem for simple
cones.  Therefore, let $\w_j$ for $j= 1, \dots, d$ be the generators of the
simple cone $\K$. Then the cone $\v + \K$ has generators $\v + \w_j$ and we have

\begin{eqnarray}
\nonumber \alpha_{\v + \K}(-s)&=& \lim_{\epsilon \to 0} \sum_{m \in \zed^d} \left(1_{\v + \K}
\ast \phi_{\epsilon} \right)(m) e^{2 \pi i \langle -s, m \rangle}\\
&=& \lim_{\epsilon \to 0} \sum_{m \in \zed^d} \widehat{\left(1_{\v + \K} \ast
\phi_{\epsilon} \right)}(m - s) \label{G} \\ \nonumber &=& \lim_{\epsilon \to 0} \sum_{m \in \zed^d} \hat{1}_{\v + \K}(m - s) \hat{\phi_{\epsilon}}(m - s) .
\end{eqnarray}
We used Poisson summation in the (\ref{G}) above and we note that the formula for the Fourier-Laplace transform of the shifted cone $\v+\K$ is obtained from that of $\K$, since $\hat{1}_{\v+\K} = \hat{1}_{\K} \cdot e^{2 \pi i \langle \v, \, \cdot \, \rangle}$. Thus

\begin{eqnarray*}
\alpha_{\v + \K}(-s)&=& \lim_{\epsilon \to 0} \sum_{m \in \zed^d} \frac{(-2 \pi i)^{- \, d} |\det{\K}| \ e^{2 \pi i \langle \v,\; m - s
\rangle}}{\prod_{j=1}^{d} \langle \w_j, m-s \rangle} \ \hat{\phi_{\epsilon}}(m -s) \\ &=& \lim_{\epsilon \to 0} \sum_{n \in \zed^d} \frac{(-2 \pi i )^{- \, d}|\det{\K}| \ e^{2 \pi i \langle \v,\; - n - s \rangle}}{\prod_{j=1}^{d} \langle \w_j, -n-s \rangle} \ \hat{\phi_{\epsilon}}(-n - s) \\ &=& \!\! (-1)^d \lim_{\epsilon \to 0} \!\! \sum_{n \in \zed^d} \! \frac{(-2 \pi i )^{- \, d} |\det{\K}| \ e^{2 \pi i \langle - \v,\; n + s \rangle}}{\prod_{j=1}^{d} \langle \w_j, n+s \rangle} \ \hat{\phi_{\epsilon}}(n + s) \\ &=& (-1)^d \lim_{\epsilon \to 0} \sum_{n \in \zed^d}  \hat{1}_{-\v + \K}(n + s) \hat{\phi_{\epsilon}}(n + s) \\ &=& (-1)^d \lim_{\epsilon \to 0} \sum_{n \in \zed^d} \widehat{(1_{-\v+\K} \ast \phi_{\epsilon} )}(n + s) \\ &=& (-1)^d \lim_{\epsilon \to 0} \sum_{n \in
\zed^d} \left(1_{-\v+\K} \ast \phi_{\epsilon} \right)(n)e^{2 \pi i \langle s, n \rangle} \\ &=& (-1)^d \alpha_{-\v+ \K}(s).
\end{eqnarray*}
We again used the fact that the lattice sum is invariant under the substitution $m = - n$ and that $\hat{\phi_\epsilon} (-z) = \hat{\phi_\epsilon} (z)$, for all $z \in \C^d$.
\end{proof}

\section{A Brion-type theorem for  solid angles sums over real polytopes}

Here we state and prove the real analogue of Brion's theorem for rational polytopes, in terms of generalized solid angles.   We note that the finite sum 
\[
\alpha_{\P} (s) = \sum_{m \in \, \P \cap \, \zed^d} \omega_\P (m) e^{2 \pi i \langle s, \, m\rangle}
\] 
can be construed as
a discrete volume measure of $\P$, since   $\alpha_{\P} (s) $ assigns to each integer point in the interior of $\P$ a weight of $1$ and to each boundary integer point of $\P$  a weight between $0$ and $1$.    
The following theorem transfers the computation of a finite sum over a polytope to a finite collection of the infinite vertex tangent cone sums  $\alpha_{\K_\v}(s)$.

\begin{thm}
Suppose $\P$ is any real, convex d-polytope.  Then we have the following identity of
meromorphic functions for $s \in \C^d$:
\begin{equation}
\alpha_{\P} (s) = \sum_{\substack{\v \; \textrm{a vertex}\\\textrm{ of} \; \P}}
\alpha_{\K_{\v}}(s),
\end{equation}
where $\K_{\v}:= \{\v + \lambda(\y - \v): \, \y \in \P, \lambda \in \real_{\geq 0} \}$ is the vertex 
tangent cone of $\P$ at the vertex $\v$.
\end{thm}
\begin{proof}
We begin with the Brianchon-Gram identity \cite{Sinai}:
\begin{equation}
1_{\P} (\x) = \sum_{\F \subseteq \P} (-1)^{\textrm{dim} \, \F} 1_{\K_{\F}}(\x),
\end{equation}
where the sum is taken over all nonempty faces $\F$ of $\P$ and $\K_{\F}$ is the tangent cone attached to $\F$ defined by $\K_{\F} := \{\mathbf{x} + \lambda ( \mathbf{y} - \mathbf{x}) : \mathbf{x} \in \F, \mathbf{y} \in \P, \lambda \in \real_{\geq 0}\}$. Next, we take the convolution of both sides with $\phi_{\epsilon}$, then multiply by $z^m$, and
finally, sum over all $m \in \zed^d$ to obtain

\begin{equation}\label{brionsum}
\sum_{m \in \zed^d}(1_{\P} \ast \phi_{\ep}) (m) z^m = \sum_{m \in \zed^d}
\sum_{\F \subseteq \P} (-1)^{\textrm{dim} \, \F} (1_{\K_{\F}} \ast \phi_{\ep})
(m) z^m.
\end{equation}
We wish to take the limit as $\ep \to 0$ of both sides of equation (\ref{brionsum}), but we first note that the infinite lattice sums are absolutely convergent due to the presence of the damping function $\phi_\ep$ and hence we can take the limit inside the sum.  Thus, we obtain
\begin{eqnarray*}
 \sum_{m \in \zed^d} \omega_{\P}(m) z^m &=& \sum_{m \in \zed^d} \sum_{\F
\subseteq \P} (-1)^{\textrm{dim} \, \F} \omega_{\K_\F} (m) z^m \\
&=& \!\!\!\! \sum_{\substack{\v \; \textrm{a vertex}\\\textrm{ of} \; \P}}
\sum_{m \in \zed^d} \omega_{\K_\v} (m) z^m + \!\! \sum_{\substack{\F \subseteq
\P \\ \\ \textrm{dim} \, \F > \, 0}} (-1)^{\textrm{dim} \, \F} \sum_{m \in
\zed^d} \omega_{\K_\F} (m) z^m.
\end{eqnarray*}
With the substitution $z^m = e^{2 \pi i \langle s , \, m \rangle}$, we have
shown that
\begin{equation}
\alpha_{\P} (s) = \sum_{\substack{\v \; \textrm{a vertex}\\\textrm{ of} \; \P}}
\alpha_{\K_{\v}}(s) + \sum_{\substack{\F \subseteq \P \\ \\ \textrm{dim} \, \F > \,
0}} (-1)^{\textrm{dim} \, \F} \alpha_{\K_{\F}}(s).
\end{equation}
Therefore, it remains to show that $\alpha_{\K_{\F}}(s) = 0$ for every face $\F$ of
$\P$ with $\textrm{dim} \, \F > \, 0$.  To this end, consider such a
$\alpha_{\K_{\F}}(s)$. Since $\K_{\F}$ is also a cone, we can write $\K_{\F}$ as the
disjoint union of its relative open faces $\G^\circ$ and obtain

\begin{equation}
\alpha_{\K_{\F}}(s) = \sum_{m \in \zed^d} \omega_{\K_\F} (m)z^m =\sum_{\G \subseteq
\, \F}\sum_{m \in \zed^{d} \cap \G^\circ} \omega_{\K_\F} (m)z^m.
\end{equation}
Since $\omega_{\K_\F} (m)$ is constant on the relative interior of each face
$\G$ of $\F$, we denote $\omega_{\K_\F} (m)$ by $\omega_{\G}$ when $m \in
\G^\circ$. Then we have
\begin{equation}\label{F}
\alpha_{\K_{\F}}(s) =\sum_{\G \subseteq \, \F} \omega_{\G}  \!\! \sum_{m \in
\zed^{d} \cap \G^\circ} \!\!\! z^m.
\end{equation}
Recall that $\textrm{dim} \, \F > \, 0$, and so $\textrm{dim} \, \G > \, 0$ for
every face $\G$ of $\F$.  Therefore, $\G^\circ$ contains a line and by theorem
3.1 in \cite{Barvinok4}
\begin{equation}
\sum_{m \in \zed^{d} \cap \G^\circ} \!\!\! z^m =0.
\end{equation}
Thus, by equation (\ref{F}), $\alpha_{\K_{\F}}(s) \! = 0$ for every face $\F$ of $\P$
with $\textrm{dim} \, \F >  0$.
\end{proof}

\section{Solid Angle Reciprocity for Real Polytopes}
We now introduce a measure of discrete volume:
\begin{equation}
A_{\P}(t) := \sum_{m \in \zed^{d}} \omega_{t\P}(m),
\end{equation}
where $\omega_{t\P} (m)$ is the generalized solid angle measure at $m\in \zed^d
\cap t \P$ defined in (\ref{restateanalyticomega}). Our next theorem is a generalization of the solid angle analogue of
Macdonald's reciprocity, which states that
\begin{equation}
A_\P (t) = (-1)^{\textrm{dim} \, \P} A_\P (-t),
\end{equation}
for $t \in \zed$ and for rational convex polytopes  $\P$ \cite{Sinai}.  First, we define a generalized
function for $s \in \C^d$ by
\begin{equation}
A_{\P}(t, s) := \sum_{m \in \zed^{d}} \omega_{t\P}(m) e^{2 \pi i \langle m, \,
s \rangle}.
\end{equation}

I. G. Macdonald introduced the notation $A_\P(t)$ to denote the \emph{solid angle} measure of a polytope.   We can relate his notation to our solid angle sum $\alpha_{t\P}$ by noting that
 $A_{\P}(t, s) = \alpha_{t\P}(s)$.    For the remainder of this paper, we use Macdonald's notation to emphasize the independent variable $t$, which we extend from $t \in \zed$ to any  
 $t \in \Bbb R$.

We will show that $A_{\P}(t, s)$ is a real analytic function of $t$ which satisfies
the reciprocity relation $A_\P (-t, s) = (-1)^{\textrm{dim} \P} A_\P (t, - s)$.
Furthermore, the following proof extends Macdonald's reciprocity to all \emph{real}
convex polytopes via $A_\P (t) = \lim_{s \to 0} A_\P (t, s)$, and to all real dilations $t$.

\begin{thm}[\textbf{Generalized Macdonald's Reciprocity}]\label{Macdonald}
Suppose $\P$ is a real convex $d$-polytope in $\real^d$.  Then

\medskip

\noindent \emph{(1)}  For $t \in \real$ and $s \in \C^d$, $A_{\P} (t, s)$ satisfies
\begin{equation}
A_\P (-t, s) = (-1)^{d} A_\P (t, -s).
\end{equation}

\noindent \emph{(2)}  Furthermore, if $\P$ is a simple $d$-polytope, $t \in \real$ and $s \in \C^d$, then the continuation
of $A_{\P}(t, s)$ to a real analytic function of $t$ is given by
\begin{equation}\label{A_P}
A_{\P} (t, s) = \lim_{\ep \to 0}\sum_{\substack{\v \; \emph{a vertex}\\\emph{
of} \; \P}} \frac{|\det{\K(\v)}|}{(-2 \pi i )^{d}} \sum_{m \in \zed^d}
\frac{\exp(2 \pi i  \,   t \langle \v,\; m + s \rangle)\ \hat{\phi_{\epsilon}}(m +
s)}{\prod_{j=1}^{d} \langle \w_j(\v), m+s \rangle }.
\end{equation}
\end{thm}
\begin{proof}
Since solid angles are additive and we can assume a triangulation of a
polytope, it suffices to prove this theorem for a real simplex $\P$. We will
use the fact that
\begin{equation}
 \omega_{t\P} (m) =\lim_{\ep \to 0} \left(1_{t\P} \ast \phi_{\ep} \right)(m),
\end{equation}
for an appropriate choice of $\phi_{\epsilon}$ with $\phi_{\epsilon}(-x) =
\phi_{\epsilon}(x)$. Then we have

\begin{eqnarray}
A_{\P}(t, s) &:=& \sum_{m \in \zed^{d}} \omega_{t\P}(m) e^{2 \pi i \langle m,
\, s \rangle}\\&=& \lim_{\epsilon \to 0} \sum_{m \in \zed^d} \left(1_{t\P} \ast
\phi_{\epsilon} \right)(m)e^{2 \pi i \langle m, \, s \rangle} \\&=&
\lim_{\epsilon \to 0} \sum_{m \in \zed^d} \widehat{\left(1_{t\P} \ast
\phi_{\epsilon} \right)}(m+s)\label{thirdequation}
\\&=& \lim_{\epsilon \to 0} \sum_{m \in \zed^d} \hat{1}_{t\P}(m+s)
\hat{\phi_{\epsilon}}(m+s).
\end{eqnarray}

\noindent We used Poisson summation in the (\ref{thirdequation}).  Next, we use an extension of Brion's theorem for \textit{real} polytopes due to Barvinok
\cite{Barvinok1} to obtain
\begin{eqnarray}
A_{\P}(t, s)&=& \lim_{\epsilon \to 0} \sum_{m \in \zed^d}
\left(\sum_{\substack{\v \; \textrm{a vertex}\\\textrm{ of} \; \P}}
\hat{1}_{t\v + \K(\v)}(m+s) \right) \hat{\phi_{\epsilon}}(m+s).
\end{eqnarray}
Barvinok's   theorem \cite{Barvinok1}  allows us to write $\hat{1}_{t\P}$ as the sum of Fourier-Laplace transforms over the tangent cones at the vertices of $t\P$.  Therefore, if $\v + \K(\v)$ is the tangent cone at the vertex $\v$ of $\P$, where $\K(\v)$ is a simple cone with apex at the origin, then $t(\v+\K(\v)) = t\v + \K(\v)$ is the tangent cone at the vertex $t\v$ of $t\P$, since a cone whose apex is the origin does not change under dilation. Using the formula for the Fourier-Laplace transform of a simple cone
\begin{eqnarray}
\!\!\!\!\!\!\!\!\!\!\!\! A_{\P}(t, s) \!\!\!\! &=& \!\!\!\! \lim_{\epsilon \to 0} \!\! \sum_{m \in \zed^d} \!\! \left(\sum_{\substack{\v \; \textrm{a} \\ \textrm{vertex}\\\textrm{ of} \; \P}} \frac{ |\det{\K(\v)}| \ \exp(2 \pi i \langle \, t\v,\; m+s \rangle)}{(-2 \pi i)^{d}\prod_{j=1}^{d} \langle \w_{j}(\v), m+s \rangle} \! \right) \! \hat{\phi_{\epsilon}}(m+s)
\\ &=& \!\!\!\! \lim_{\epsilon \to 0} \sum_{\substack{\v \; \textrm{a} \\ \textrm{vertex}\\\textrm{ of} \; \P}}  \frac{ |\det{\K(\v)}|}{(-2 \pi i )^{d}}\sum_{m \in \zed^d} \frac{ \ \exp(2 \pi i t\langle \, \v,\; m+s \rangle)
\hat{\phi_{\epsilon}}(m+s)}{\prod_{j=1}^{d} \langle \w_{j}(\v), m+s \rangle}.
\end{eqnarray}
We note that the only place a $t$ appears in this last equation is in the
exponent of an exponential. Hence, $A_{\P}(t, s)$ is a real analytic function of $t$, because we can differentiate inside the summation sign due to the rapid  convergence provided by $\hat{\phi_\ep}$.  This proves part (2).

Now for the proof of part (1), we evaluate the continuation of $A_{\P}(t, s)$ at
$-t$ to obtain
\begin{eqnarray}
\!\!\!\!\!\!\!\!\!\!\!\!\!\!\!\!A_{\P}(-t, s) \!\!\!\! &=& \lim_{\epsilon \to 0} \sum_{\substack{\v \; \textrm{a vertex}\\
\textrm{ of} \; \P}}  \frac{ |\det{\K(\v)}|}{(-2 \pi i )^{d}}\sum_{m \in
\zed^d} \frac{ \ e^{2 \pi i(-t)\langle \, \v,\; m+s \rangle}
\hat{\phi_{\epsilon}}(m+s)}{\prod_{j=1}^{d} \langle \w_{j}(\v), m+s \rangle}\\
&=& \lim_{\epsilon \to 0} \sum_{\substack{\v \; \textrm{a vertex}\\
\textrm{ of} \; \P}}  \frac{ |\det{\K(\v)}|}{(-2 \pi i )^{d}}\sum_{n \in
\zed^d} \frac{ \ e^{2 \pi i t\langle \, \v,\; n-s \rangle}
\hat{\phi_{\epsilon}}(-n+s)}{\prod_{j=1}^{d} \langle \w_{j}(\v), -n+s \rangle}\\ &=&  (-1)^d \lim_{\epsilon \to 0} \sum_{\substack{\v \; \textrm{a vertex}\\ \textrm{ of} \; \P}}  \frac{ |\det{\K(\v)}|}{(-2 \pi i )^{d}}\sum_{n \in \zed^d} \frac{ \ e^{2 \pi i t\langle \, \v,\; n-s \rangle}
\hat{\phi_{\epsilon}}(n-s)}{\prod_{j=1}^{d} \langle \w_{j}(\v), n-s \rangle}\\
&=& (-1)^d A_{\P}(t, -s).
\end{eqnarray}
We again used the fact that the lattice sum is invariant under the substitution $m = - n$ and that $\hat{\phi_\epsilon} (-z) = \hat{\phi_\epsilon} (z)$, for all $z \in \C^d$.
\end{proof}

\smallskip
\begin{cor}
Suppose $\P$ is a real convex $d$-polytope in $\real^d$ with $d$ odd.  Then
\[
A_\P (0, 0) = 0.
\]
\end{cor}
\begin{proof}
By Theorem \ref{Macdonald}, we have
\[
A_\P (0, 0) = (-1)^{d} A_\P (0, 0) = - A_\P (0, 0).
\]
\end{proof}

We pause for a moment to discuss the subtlety involved in computing $A_{\P}(t)$
using the previous theorem.  We know that $A_{\P}(t)$ is a real analytic function of
$t$ and in fact is a quasi-polynomial in $t \in \zed$ when $\P$ is a rational polytope \cite{Macdonald}.
The introduction of the complex parameter $s$ in $A_{\P}(t, s)$ prevents the
denominators of $\hat{1}_{\v+\K(\v)}$ from being zero.  So one might wonder if
$A_\P (t) = \lim_{s \to 0} A_\P (t, s)$ even exists.  It is Barvinok's extension of Brion's theorem
that tells us that when we add up $\hat{1}_{\v+\K(v)}(m+s)$ at every vertex $\v$, magically all of the singularities in $s \in \C^d$ cancel.

To compute $A_{\P}(t)$ from (\ref{A_P}), we write all of the rational functions on the right-hand side over one denominator and use L'H\^opital's rule to compute the limit as $s \to 0$. The following example will illustrate this procedure.

\begin{ex}
\noindent \emph{Let $\P$ be the triangle in $\real^2$ with vertices $\v_1 =(0,0),
\v_2 = (0,1)$ and $\v_3 = (\sqrt{3}, 0)$}.
\begin{center}
\begin{figure}[h]
\setlength{\unitlength}{3.6cm}
\begin{picture}(1,1)
\thicklines
\put(.3,0){\line(0,1){1}}
\put(.3,0){\line(1,0){2}}
\put(.3,1){\line(2,-1){2}}
\put(.3,-.1){$\v_1 = (0,0)$}
\put(.3,1.03){$\v_2 = (0,1)$}
\put(2.2,-.1){$\v_3 = (\sqrt{3},0)$}
\end{picture}
\\
\caption{The triangle $\P$.}
\end{figure}
\end{center}
\end{ex}
\noindent To calculate $A_{\P} (t)$, we use equation (\ref{A_P}) in Theorem \ref{Macdonald} and we begin by evaluating the determinant of the tangent cone at each vertex.  We have
 \[
 |\det{\K(\v_1)}| = \det\left(\begin{array}{ccc} 1 & 0 \\ 0 & 1 \end{array}\right)= 1,
 \]
 \[
 |\det{\K(\v_2)}| = \det\left( \! \begin{array}{ccc} 0 & \sqrt{3} \\ -1 & -1 \end{array} \!\right)= \sqrt{3},
 \]
 \[
 \textrm{and }|\det{\K(\v_3)}| = \det\left(\begin{array}{ccc} -\sqrt{3} & -1 \\ 1 & 0 \end{array}\right)= 1.
 \]
We also need to evaluate
\begin{equation}\label{expfraction}
\frac{e^{2 \pi i t \langle \v,\; m + s \rangle} } {\prod_{j=1}^{2} \langle \w_j(\v), m+s \rangle},
\end{equation}
for each of the vertex cones $\K_{\v_1}, \K_{\v_2}$, and $\K_{\v_3}$.
Then (\ref{expfraction}) equals
\[
\frac{1}{(m_1 + s_1)(m_2 +s_2)}, \ \ \ \frac{-e^{2 \pi i t (m_2 + s_2)}}{(m_2 +s_2)(\sqrt{3}(m_1+s_1) - m_2 -s_2)}, \ \ \textrm{and } \;
\]
\[
 \frac{e^{2 \pi i t \sqrt{3}(m_1+s_1)}}{(m_1 +s_1)(\sqrt{3}(m_1+s_1) - m_2 -s_2)}, \ \textrm{for $\v_1, \v_2,$ and $\v_3$ respectively.}
\]

\noindent Thus
\begin{eqnarray*}
A_{\P} (t, s) &=& \lim_{\ep \to 0}\sum_{\substack{\v \; \textrm{a
vertex}\\\textrm{ of} \; \P}} \frac{|\det{\K(\v)}|}{(-2 \pi i )^{2}} \sum_{m
\in \zed^2} \frac{e^{2 \pi i t \langle \v,\; m + s \rangle}\
\hat{\phi_{\epsilon}}(m + s)}{\prod_{j=1}^{2} \langle \w_j(\v), m+s \rangle }\\
&=& \lim_{\ep \to 0} \frac{1}{- 4 \pi^{2}} \sum_{(m_1 , m_2) \in \zed^2}
\hat{\phi_{\epsilon}}(m + s) \left( \frac{1}{(m_1 + s_1)(m_2 +s_2)} \right. \\
&-& \left. \!\!\! \frac{\sqrt{3} e^{2 \pi i t (m_2 +s_2)}}{(m_2
+s_2)(\sqrt{3}(m_1+s_1) - m_2 -s_2)} + \frac{e^{2 \pi i t \sqrt{3}(m_1
+s_1)}}{(m_1 +s_1)(\sqrt{3}(m_1+s_1) - m_2 -s_2)} \! \right) \\
&=& \lim_{\ep \to 0} \frac{1}{- 4 \pi^{2}} \sum_{(m_1 , m_2)  \in \zed^2}
\hat{\phi_{\epsilon}}(m + s) \cdot \frac{f(t, s)}{g(t, s)},
\end{eqnarray*}
where
\[
\frac{f(t, s)}{g(t, s)}= \frac{\sqrt{3}(m_1+s_1) - m_2 -s_2 -\sqrt{3}(m_1+s_1)
e^{2 \pi i t (m_2 +s_2)} +(m_2 +s_2)e^{2 \pi i t \sqrt{3}(m_1+s_1)}}{(m_1 +
s_1)(m_2 +s_2)(\sqrt{3}(m_1+s_1) - m_2 -s_2)}.
\]
All that remains is to use L'H\^opital's rule to calculate
\[
 \lim_{s \to 0} \frac{f(t, s)}{g(t, s)}.
\]
In order to take the derivative with respect to $s$, we first let $s = \sigma
(x_1, x_2)$ for some fixed $(x_1, x_2) \neq 0$ and then take the derivative
with respect to $\sigma$. Since $t$ appears in an exponential in the numerator,
each iteration of L'H\^opital's rule will produce a factor of $t$ in the
numerator. It is known that for a rational $d$-polytope, $A_\P (t)$ is a
quasi-polynomial in $t$ of degree $d$.  Therefore, in general, one must apply
L'H\^opital's rule $d$ times for a $d$-polytope. Thus

\begin{eqnarray*}
\!\!\!\!\!\!\!\!\!  \lim_{s \to 0} \frac{f(t, s)}{g(t, s)} \!\!\!&=&
\lim_{\sigma \to 0} \frac{f(t, \sigma)}{g(t, \sigma)} \\&=& \lim_{\sigma \to 0}
\frac{f'(t, \sigma)}{g'(t, \sigma)}  \\&=& \lim_{\sigma \to 0}
\frac{f''(t, \sigma)}{g''(t, \sigma)} \\&=& \frac{f''(t, 0)}{g''(t, 0)} \\
&=& \!\!\!\! \frac{-6 \pi^2 m_2 x_1^2 t^2 e^{2 \pi i t \sqrt{3}m_1}    + 2 \pi
i \sqrt{3} x_1 x_2 t ( e^{2 \pi i t \sqrt{3}m_1} - e^{2 \pi i t m_2}) + 2 \pi^2 \sqrt{3} m_1 x_2^2 t^2 e^{2 \pi i t m_2} }{-x_2(2m_2 x_1 + m_1 x_2) +
\sqrt{3}x_1 (m_2 x_1 + 2m_1 x_2)},
\end{eqnarray*}
where we used Mathematica in these last steps. We can now choose $(x_1, x_2)$ to be any non-zero vector as long as the denominator is never zero.  Therefore, we let $(x_1, x_2) = (1, 1)$ and we have
\begin{eqnarray*}
A_\P (t) &=& \lim_{s \to 0} A_\P (t, s)\\
&=&\lim_{\ep \to 0} \frac{1}{- 4 \pi^{2}} \sum_{(m_1 , m_2)  \in \zed^2}
\hat{\phi_{\epsilon}}(m) \cdot  \frac{f''(t, 0)}{g''(t, 0)} \\
&=&\lim_{\ep \to 0} \frac{1}{- 4 \pi^{2}} \sum_{(m_1 , m_2)  \in \zed^2}
\hat{\phi_{\epsilon}}(m) \cdot \\
&& \frac{-6 \pi^2 m_2 t^2 e^{2 \pi i t \sqrt{3}m_1} + 2 \pi i \sqrt{3} t ( e^{2
\pi i t \sqrt{3}m_1} - e^{2 \pi i t m_2}) + 2 \pi^2 \sqrt{3} m_1 t^2 e^{2 \pi i
t m_2} }{-2m_2 - m_1 + \sqrt{3}(m_2 + 2m_1)}.
\end{eqnarray*}
\\
When
\begin{equation}
\phi_{\ep}(s) = \ep^{-\frac{d}{2}} \exp \left( \frac{- \pi}{\ep} \langle s, \, s \rangle \right) = \ep^{-1} \exp \left( \frac{- \pi}{\ep} \left(s_1^2 + s_2^2 \right) \right),
\end{equation}
it follows that
\begin{equation}\label{phi_ep}
\hat{\phi_{\ep}}(m_1, m_2) = \ep^{-\frac{1}{2}} \exp \left( - \pi \ep \left( m_1^2 + m_2^2 \right) \right).
\end{equation}
Since $\hat{\phi_{\ep}}(m_1, m_2)$ provides absolute convergence, we can break up the series for $A_\P (t)$ and use equation (\ref{phi_ep}) to obtain the following:

\begin{eqnarray*}
A_\P (t) 
&=& t^2 \left(\lim_{\ep \to 0} \frac{\ep^{-\frac{1}{2}} }{- 4 \pi^{2}} \sum_{(m_1 , m_2)  \in \zed^2} \frac{ \left( 2 \pi^2 \sqrt{3} m_1 e^{2 \pi i t m_2}-6 \pi^2 m_2 e^{2 \pi i t \sqrt{3}m_1} \right) e^{- \pi \ep \left( m_1^2 + m_2^2 \right)}}{-2m_2 - m_1 + \sqrt{3}(m_2 + 2m_1)} \right) \\
&+& t \left(\lim_{\ep \to 0} \frac{\ep^{-\frac{1}{2}} }{- 4 \pi^{2}} \sum_{(m_1 , m_2)  \in \zed^2} \frac{2 \pi i \sqrt{3} ( e^{2
\pi i t \sqrt{3}m_1} - e^{2 \pi i t m_2}) e^{- \pi \ep \left( m_1^2 + m_2^2 \right)}}{-2m_2 - m_1 + \sqrt{3}(m_2 + 2m_1)} \right).  
\end{eqnarray*}
\begin{flushright}
\qedsymbol
\end{flushright}

In the previous example, we note that $A_\P (0) = 0$ and the dimension, $d = 2$, is even. This leads to the following conjecture:
\begin{conj}
Suppose $\P$ is a real convex $d$-polytope for any dimension $d$.  Then
\[
A_\P (0, 0) = 0.
\]
\end{conj}

To summarize, we have extended Macdonald's solid angle function $A_{\P}(t)$ to all
real polytopes $\P$ and all real dilations $t$, using $A_\P(t,s)$.    It now becomes an interesting question to look for special values of the discrete volume function  $A_\P(t,s)$ for various values
of $t$ and $s$.   

\bibliographystyle{plain}

 \sc     
 \noindent   Department of Mathematics, Physics, and Computer Science\\ 
          Georgetown College \\ 
          Georgetown, KY  40324  \\ 
 {\tt   David\underline{ }DeSario@georgetowncollege.edu  }\\

 \noindent  Department of Mathematics\\
 Temple University\\
 Philadelphia, PA 19122 \\
 {\tt srobins@temple.edu}

\end{document}